\documentclass[12pt]{article}

\usepackage[english]{babel}
\usepackage{enumerate}
\usepackage{xcolor}
\usepackage{authblk}
\usepackage{bbm}
\usepackage[margin=1in]{geometry}
\usepackage{enumitem}

\usepackage{amsmath}
\usepackage{amsfonts}
\usepackage{amsthm}
\usepackage{mathtools}
\usepackage{graphicx}
\usepackage{tikz}
\usetikzlibrary{shadows, arrows.meta, cd}

\usepackage{hyperref}
\hypersetup{
	colorlinks,
	linkcolor={red!90!black},
	citecolor={blue!90!black},
	urlcolor={blue!80!black}
}

\xdefinecolor{olive}{cmyk}{0.64,0,0.95,0.4}
\xdefinecolor{blueberry}{cmyk}{0.44,0.23,0.0,0.18}
\xdefinecolor{maroon}{cmyk}{0.00,1.0,1.0,0.498}

\theoremstyle{plain}
\newtheorem{thm}{Theorem}[section]
\newtheorem{theorem}[thm]{Theorem} 
\newtheorem{lemma}[thm]{Lemma} 
\newtheorem{proposition}[thm]{Proposition}
\newtheorem{corollary}[thm]{Corollary}

\theoremstyle{definition}
\newtheorem{remark}[thm]{Remark}

\newtheorem{conjecture}[thm]{Conjecture}

\newtheorem{definition}[thm]{Definition}
\numberwithin{equation}{section}

\newcommand{\ignore}[1]{}

\DeclareMathOperator{\diag}{diag}

\newcommand{\bz}{\mathbf{z}}
\newcommand{\bw}{\mathbf{w}}
\newcommand{\be}{\mathbf{e}}
\newcommand{\bc}{\mathbf{c}}
\newcommand{\bv}{\mathbf{v}}

\newcommand{\cF}{\mathcal{F}}

\title{Threshold graphs, Kemeny's constant, and related random walk parameters}
\author{Jane Breen\textsuperscript{1}, Sooyeong Kim\textsuperscript{2}\footnote{Contact: kimswim@yorku.ca}, Alexander Low Fung\textsuperscript{3}, Amy Mann\textsuperscript{4},\\ Andrei A. Parfeni\textsuperscript{5} and Giovanni Tedesco\textsuperscript{6}}

\begin{document}

\maketitle

\begin{abstract}
    Kemeny's constant measures how fast a random walker moves around in a graph. Expressions for Kemeny's constant can be quite involved, and for this reason, many lines of research focus on graphs with structure that makes them amenable to more in-depth study (for example, regular graphs, acyclic graphs, and 1-connected graphs). In this article, we study Kemeny's constant for random walks on threshold graphs, which are an interesting family of graphs with properties that make examining Kemeny's constant difficult; that is, they are usually not regular, not acyclic, and not 1-connected. This article is a showcase of various techniques for calculating Kemeny's constant and related random walk parameters for graphs. We establish explicit formulae for $\mathcal{K}(G)$ in terms of the construction code of a threshold graph, and completely determine the ordering of the accessibility indices of vertices in threshold graphs.
\end{abstract}



\addtocounter{footnote}{1}
\footnotetext{Faculty of Science, Ontario Tech University, Oshawa, Ontario, Canada}

\addtocounter{footnote}{1}
\footnotetext{Department of Mathematics and Statistics, York University, Toronto, Ontario, Canada}

\addtocounter{footnote}{1}
\footnotetext{Department of Mathematics, San Francisco State University, San Francisco, California, USA}
\addtocounter{footnote}{1}
\footnotetext{Department of Mathematics, University of Toronto, Toronto, Ontario, Canada}
\addtocounter{footnote}{1}
\footnotetext{Department of Mathematics, Yale University, New Haven, Connecticut, USA}
\addtocounter{footnote}{1}
\footnotetext{Department of Mathematical and Computational Sciences, University of Toronto Mississauga, Mississauga, Ontario, Canada}

\section{Introduction}

Random walks are frequently used to uncover qualitative and quantitative information about the structure of graphs. Given a graph $G$ with vertices $V = \{v_1, v_2, \ldots, v_n\}$, a random walk on the vertices of $G$ is a Markov chain in which the state space is the vertex set $V$, and the transition probability $t_{i,j}$ of moving from the $i^{th}$ state to the $j^{th}$ state is given by 
\[t_{i,j} = \left\{\begin{array}{cc} \frac{1}{\deg(v_i)}, & \mbox{if $v_i$ is adjacent to $v_j$};\\
0, & \mbox{otherwise.}\end{array}\right.\]
That is, a random walk on a graph can be conceptualized as a dynamical process in which a `random walker' occupies a vertex of the graph, and at each time-step, chooses an adjacent vertex to their current one uniformly at random and moves to it. By exploring the behaviour of this stochastic process, qualities of the graph can be uncovered. For example, the mixing properties of the corresponding Markov chain are associated with \emph{expansion properties} of the graph (see \cite{fanchung}).

Kemeny's constant is a Markov chain parameter which has recently been the subject of many research papers in the context of its role as a graph invariant measuring how `well-connected' the graph is (see, for example, \cite{kirkland2016kemeny} for Braess edges, \cite{altafini2023edge} for edge centrality score via Kemeny's constant and \cite{kim2023bounds} for the Nordhaus-Gaddum problem regarding Kemeny's constant). Loosely speaking, Kemeny's constant for a graph, denoted $\mathcal{K}(G)$, can be interpreted as the expected length of a random trip between vertices of $G$ (a formal definition is given in the next section). It can also be derived from the eigenvalues of the transition matrix $T$ for the random walk on the graph, meaning it is a spectral graph invariant. 

Kemeny's constant can be quite difficult to work with and results in this area has been constrained by this, with many lines of research focusing on graphs with structure that makes them amenable to more in-depth study. Divide-and-conquer expressions for $\mathcal{K}(G)$ are given in \cite{breen2022kemeny, faught20221} for graphs with a cut vertex or a bridge. An expression for $\mathcal{K}(G)$ in terms of spanning trees and forests of $G$ is given in \cite{kirkland2016kemeny}, and as such Kemeny's constant has been well-studied in the case that $G$ is a tree (see \cite{ciardo2022kemeny, jang2023kemeny, kirkland2016kemeny}), since those formulas become more straightforward in the acyclic case. Families of large-diameter graphs are considered in \cite{breen2019computing}; considering such families allows for small changes in the value of $\mathcal{K}(G)$ to be ignored in pursuit of general results about relationships between $\mathcal{K}(G)$ and other structural features of the graph. Note that the maximum order of magnitude for $\mathcal{K}(G)$ is $\mathcal{O}(n^3)$ for a graph $G$ of order $n$, and the conjectured extremal graph is a \emph{barbell graph}. Finally, results for Kemeny's constant for regular graphs are easily derived from existing spectral results for graph matrices other than the probability transition matrix, such as the Laplacian or adjacency matrix (see \cite{palacios2011broder}). 

In this article, we consider the task of calculating Kemeny's constant for \emph{threshold graphs}, an interesting family of graphs which are not typically sparse or acyclic, not regular, and have diameter 2 in the connected case (which we exclusively consider). In this article, we prove several results concerning the range of values of $\mathcal{K}(G)$ in the case that $G$ is a threshold graph, including how Kemeny's constant may be directly computed from the so-called \emph{construction code} for the threshold graph. In addition to the contribution of these results, this article should be considered as a showcase of various techniques to calculate Kemeny's constant: the difficulty of working with $\mathcal{K}(G)$ in the case that the graph family $G$ is not constrained as described above is highlighted in this article, even as we produce complex and involved expressions for $\mathcal{K}(G)$ using surprisingly elegant lemmas.

In Section 2, we introduce Kemeny's constant formally, as well as additional mathematical preliminaries. In Section 3, we consider the construction code of a threshold graph and how to derive Kemeny's constant from this alone, using a particular unitary matrix which can be shown to diagonalize the Laplacian matrix of any threshold graph, once its rows and columns are arranged with respect to the construction code. At the end of the section, we demonstrate how these results are implemented for a specific family of threshold graphs conjectured to give the maximum value of $\mathcal{K}(G)$, and use the  reduction in computation complexity for $\mathcal{K}(G)$ provided by these results to confirm the conjecture up to $n=21$. In Section 4, we consider combinatorial expressions for Kemeny's constant and related parameters, and showcase how those may be calculated for threshold graphs from their unique structure.

\section{Preliminaries}

\subsection{Graphs and matrices}
Given a graph $G$ with vertices $\{v_1, v_2, \ldots, v_n\}$, the \emph{adjacency matrix} of $G$ is the $n\times n$ matrix $A=[a_{i,j}]$, whose rows and columns are indexed by the vertices of $G$ and whose entries are defined
\[a_{i,j}= \left\{\begin{array}{cc} 1, & \mbox{if $v_i$ is adjacent to $v_j$};\\
0, & \mbox{otherwise.}\end{array}\right.\]
The \emph{degree matrix} of a graph is a diagonal matrix $D$ whose $i^{th}$ diagonal entry is $d_i:=\deg(v_i)$. The \emph{Laplacian matrix} of a graph is defined $L=D-A$, while the random walk matrix---or the transition matrix for the random walk on $G$---is given by $D^{-1}A$. 

\subsection{Kemeny's constant}
A finite, discrete-time, time-homogeneous Markov chain can be thought of as a random process which---at any given time---occupies one out of a finite number of states index $1, 2, \ldots, n$, and transitions from one state to another in finite time-steps of fixed length. The probability of transitioning from the $i^{th}$ state to the $j^{th}$ state in a single time-step is denoted $t_{i,j}$; this probability is fixed, and does not depend on how much time has passed (\emph{time-homogeneity}). Furthermore, the assumption that the probability distribution of the states at time $k+1$ (i.e. in the next step) depends only on the probability distribution at time $k$ (i.e. the current state) is referred to as the \emph{Markov property}. With this in mind, the evolution of the random process can be seen to depend entirely on the matrix of transition probabilities $T=[t_{i,j}]$, with the $(i,j)$ entry of $T^k$ representing the probability of being in the $j^{th}$ state at time $k$, given that one starts in the $i^{th}$ state. Under certain additional assumptions (that the Markov chain is \emph{regular}, or that the transition matrix is \emph{primitive}), one can use Perron-Frobenius theory to determine that \[\lim_{k\to\infty} T^k = \mathbbm{1}w^\top,\]
where $\mathbbm{1}$ is the all-ones vector, and $w^\top$ is a positive row vector whose entries sum to 1. In particular $w$ is a probability distribution vector satisfying $w^\top T = w^\top$, indicating that $w$ is the stationary distribution of the Markov chain, and the $i^{th}$ entry of $w$ represents the long-term probability that the random process occupies the $i^{th}$ state. The short-term behaviour of the Markov chain is encapsulated in the \emph{mean first passage times}, in which $m_{ij}$ denotes the expected time to reach state $j$, given that the process starts in state $i$. This, too, may be calculated using the transition matrix $T$. We direct the interested reader to \cite{kemenysnell} for details on the computation of $m_{ij}$ and further information on Markov chains; these concepts are introduced here purely to motivate the definition of Kemeny's constant, and will not be used throughout the body of the paper. 

Given a Markov chain with transition matrix $T$, fix a starting state $i$, and consider the following quantity:
\[\kappa_i = \sum_{\substack{j=1\\j\neq i}}^n w_jm_{i,j}.\]
This quantity may be interpreted as the expected time it takes to reach a randomly-chosen destination state from a fixed starting state $i$. Astonishingly, this quantity is easily-shown to be independent of $i$; as such, it is referred to as Kemeny's constant, and denoted $\mathcal{K}(T)$. Since $\sum_i w_i = 1$, the expression can be re-written to produce 
\[\mathcal{K}(T) = \sum_{\substack{ i, j = 1\\i\neq j}}^n w_im_{i,j}w_j,\]
allowing the interpretation of Kemeny's constant as the expected length of a random trip between states of the Markov chain, where both the initial state and destination state are randomly-chosen with respect to the stationary distribution.

Let $G$ be a graph with vertex set $V(G) = \{v_1, v_2 \ldots, v_n\}$ and let $d_i = \deg(v_i)$. The random walk on the graph $G$ is a Markov chain with state space given by $V(G)$, and whose transition matrix is $T = D^{-1}A$. Note that the stationary distribution vector for the random walk on $G$ is proportional to the degree vector of $G$; that is, $w_i = \frac{d_i}{2m}$, where $m$ is the number of edges in $G$. Therefore, the relative `importance' of a vertex in the random walk on the graph is determined by its degree. Furthermore, Kemeny's constant can be interpreted in the graph context as the expected length of a random trip in the graph $G$, or a weighted average of the mean first passage times between vertices, where high-degree vertices receive a larger weight in the calculation. For a simple random walk on a graph, we denote the value of Kemeny's constant as $\mathcal{K}(G)$, and treat it as a graph invariant representing how `well-connected' the vertices of the graph are.

We now outline several alternate expressions for Kemeny's constant used in this paper. In \cite{levene2002kemeny}, it is shown that for a Markov chain with transition matrix $T$ having eigenvalues $1, \rho_2, \ldots, \rho_n$, we have
\begin{equation}\label{eq:kemll}\mathcal{K}(T) =\sum_{j=2}^n \frac{1}{1-\rho_j}.\end{equation}
Of great use when considering Kemeny's constant for random walks on graphs is the following combinatorial expression for $\mathcal{K}(G)$. 
\begin{proposition}[\cite{kirkland2016kemeny}]\label{prop:F-exp}
Let $G$ be a connected graph with vertices $v_1, v_2, \ldots, v_n$, with degree vector  $\mathbf{d} = \begin{bmatrix} d_1 & d_2 &\cdots & d_n\end{bmatrix}^\top$. Let $m$ be the number of edges of $G$, $\tau$ the number of spanning trees of $G$, and let $F = [f_{i,j}]$ be the matrix whose $(i,j)$ entry $f_{i,j}$ denotes the number of spanning forests of $G$ consisting of exactly two trees, one containing $v_i$ and the other containing $v_j$.  Then 
\[\mathcal{K}(G) = \frac{\mathbf{d}^\top F \mathbf{d}}{4m\tau}.\] 
\end{proposition}

Let $R$ be the matrix given by $R = [r_{i,j}]$, where $r_{i,j}$ is the so-called \textit{effective resistance} (see \cite{bapat2010graphs}) between vertices $v_i$ and $v_j$. The quantity $r_{i,j}$ is given by $r_{i,j}=\ell_{i,i}^\dagger + \ell_{j,j}^\dagger - 2\ell_{i,j}^\dagger$ where $L^\dagger$ is the Moore-Penrose inverse of the Laplacian matrix $L$ of $G$, which is given by $L = D-A$. It appears in \cite{shapiro1987electrical} that $r_{i,j} = \frac{f_{i,j}}{\tau}$. Hence, we have 
\[\mathcal{K}(G) = \frac{\mathbf{d}^\top R \mathbf{d}}{4m}.\] 

While the relationship between Kemeny's constant and the eigenvalues of Laplacian matrix for a graph $G$ was originally discussed in \cite{palacios2011broder} in the case that $G$ is regular, the general case was considered in \cite{wang2017kemeny}, in which the authors give the following expression.

\begin{proposition}[\cite{wang2017kemeny}]\label{prop:L-exp}
Let $G$ be a connected graph of order $n$ with $m$ edges, degree vector $\mathbf{d}$, and Laplacian matrix $L$. Then 
\[\mathcal{K}(G) = \mathbf{q}^\top \mathbf{d} - \frac{\mathbf{d}^\top L^\dagger \mathbf{d}}{2m},\]
where 
$\mathbf{q}^\top$ is the column vector of diagonal entries of $L^\dagger$; $\mathbf{q}^\top = \begin{bmatrix} \ell_{1,1}^\dagger & \ell_{2,2}^\dagger & \cdots & \ell_{n,n}^\dagger\end{bmatrix}$.
\end{proposition}

 
We finish this section by introducing one final random walk parameter, conceptually linked to Kemeny's constant but serving a different role in the analysis of graphs. 
Given a Markov chain with stationary vector $w$, and mean first passage times $m_{i,j}$, fix an index $j$ and define 
\[\alpha(j) = \sum_{\substack{i=1\\i\neq j}}^n w_im_{i,j}.\]
This may be interpreted as the expected length of time to reach the $j^\text{th}$ state, where the starting state is chosen at random with respect to the stationary distribution. As such, $\alpha(j)$ is referred in \cite{kirkland2016random} to as the \emph{accessibility index} of state $j$, and in the context of random walks on graphs, can be considered as a centrality measure of the vertices in the graph. Indeed, the accessibility index is usually studied in network science literature in alternate form as the \emph{random walk centrality} \cite{noh2004random}, which is given by the reciprocal of the accessibility index, so that large values correspond to more central vertices in the graph. Additionally, note that $\mathcal{K}(G) = \sum_{j=1}^n w_j\alpha(j) - 1$.

We introduce a quantity related to the accessibility index. The moment $\mu(j)$ of a vertex $j$ in a graph $G$ is defined \cite{ciardo2022kemeny} as follows:
$$\mu(j) = \sum_{i\neq j} d_ir_{i,j}.$$
It turns out from \cite{breen2022kemeny} that they are related through the following formula:
\[
\mu(j) = \alpha(j) + \mathcal{K}(G).
\]
As $r_{i,j}$ can be obtained from the combinatorial quantities $f_{i,j}$ and $\tau$, this formula will be used in Section 4 when two accessibility indices are compared.





\subsection{Threshold graphs}
The class of graphs we consider in this paper are \emph{threshold graphs}. There are many equivalent definitions of threshold graphs (see \cite{thresholdbook}). We will primarily work with the following definition:

\begin{definition}
A {\em threshold graph} is a graph that can be constructed from a single vertex by repeatedly adding an isolated vertex or a dominating vertex.
\end{definition}

In particular, such a graph $G$ can be represented uniquely by a binary string $c_1c_2 \cdots c_n$ called the \emph{construction sequence} or \emph{construction code} of the threshold graph $G$, in which $c_i=0$ indicates that the $i^{th}$ vertex added is an isolated vertex, and $c_i=1$ if the $i^{th}$ vertex added is a dominating one. (As an example, see the threshold graph $G$ with construction sequence $01100011$ in Figure~\ref{fig:ex1}). Our convention in this article is to represent $v_1$---the initial vertex of the graph---with a 0 at the beginning of the sequence. Note that $G$ is connected only if $c_n=1$, meaning that $G$ has a dominating vertex, and the diameter of the graph is 2. We note that Kemeny's constant is undefined for disconnected graphs (and is usually taken to be infinity); we consider only connected threshold graphs in this article. Since there is a one-to-one correspondence between threshold graphs and the construction sequence, we can conclude that there are $2^{n-2}$ connected threshold graphs of order $n$.

When $0$ (resp. $1$) appears $s$ times (resp. $t$ times) consecutively in a construction code, we use $\mathbf{0}^s$ (resp. $\mathbf{1}^t$) to denote those zeros (resp. those ones). For the example in Figure~\ref{fig:ex1}, the construction code can be written as $0\mathbf{1}^2\mathbf{0}^3\mathbf{1}^2$.

The eigenvalues associated with threshold graphs are well-studied; see \cite{hammer1996laplacian} for the seminal work regarding the Laplacian spectrum of a threshold graph. For the eigenvalues of the normalized Laplacian and the random walk matrix, see \cite{banerjee2017normalized}, which makes use of quotient matrix techniques and equitable partitions to compute the eigenvalues associated with $D^{-1/2}(D-A)D^{-1/2}$ and $D^{-1}A$. Various spectral graph invariants have also been considered for threshold graphs, such as the energy \cite{jacobs2015eigenvalues}. In many cases, both spectral and structural information is directly related to the construction code of $G$; see, for example, \cite{bapat2013adjacency} where the inertia of the adjacency matrix is determined by the number of ones and zeros in the construction code, and \cite{hammer1996laplacian} where the number of spanning trees and the eigenvalues themselves are determined directly from the construction code.


\begin{figure}
\begin{center}
\begin{tikzpicture}[every node/.style = {inner sep = 1.2pt}, scale =0.8]
\foreach \x in {1,...,8} \node (\x) at ({2*cos(22.5+45*(\x-1))}, {2*sin(22.5+45*(\x-1))}) {};
\foreach \x in {1, ..., 6} \draw[thick] (8)--(\x)--(7);
\draw[thick] (7)--(8);
\draw[thick] (1)--(2)--(3)--(1);
\foreach \x in {1, ..., 8} \draw[black, fill=black] (\x) circle(3pt);
\draw (1) node [anchor=south west] {$v_1$};
\draw (2) node [anchor=south west] {$v_2$};
\draw (3) node [anchor=south east] {$v_3$};
\draw (4) node [anchor=south east] {$v_4$};
\draw (5) node [anchor=east] {$v_5$};
\draw (6) node [anchor=north east] {$v_6$};
\draw (7) node [anchor=north west] {$v_7$};
\draw (8) node [anchor=north west] {$v_8$};
\end{tikzpicture}
\end{center}
\caption{The threshold graph corresponding to the construction sequence $01100011$.}\label{fig:ex1}
\end{figure}
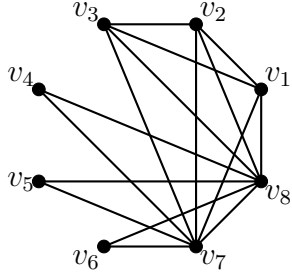



\section{A formula for Kemeny's constant directly from the construction code}
\label{sec:formula}

In this section, we derive an involved formula for Kemeny's constant of a threshold graph using only the construction code. This result uses the Laplacian formulation of Kemeny's constant in Proposition~\ref{prop:L-exp}, and hinges on a surprising result that the Laplacian matrix of every threshold graph may be diagonalized by a single common unitary matrix, which we find in Section \ref{subsec:diag}. This follows from the (equally surprising) fact that, when written with respect to an appropriate ordering of the vertices, the Laplacian matrices of any two threshold graphs commute. 

\subsection{Diagonalizing the Laplacian matrix of a threshold graph}
\label{subsec:diag}

We begin with introducing a well-known result. 

\begin{theorem}[{\cite[Theorem 2.5.5]{hornjohnson}}]\label{thm:universal and unique}
	Let $\mathcal{A}$ be a set of some $n\times n$ normal matrices. Then every pair of matrices in $\mathcal{A}$ commutes if and only if all matrices are simultaneously unitarily diagonalizable. Moreover, for any given matrix $A_0\in\mathcal{A}$ and for any given ordering $\lambda_1,\dots,\lambda_n$ of the eigenvalues of $A_0$, there is an $n\times n$ unitary matrix $U$ such that $U^{-1}A_0U = \mathrm{diag}(\lambda_1,\dots,\lambda_n)$ and $U^{-1}BU$ is diagonal for every $B\in \mathcal{A}$. 
\end{theorem}

This theorem further suggests that if there exists a matrix $A$ in $\mathcal{A}$ such that all eigenvalues are simple, then a matrix whose columns consist of normalized eigenvectors of $A$ diagonalizes every matrix in $\mathcal{A}$; furthermore, it is unique up to permutation of columns and up to signs of columns.

In this subsection, we shall show that the set of Laplacian matrices of all threshold graphs with construction code on $n$ vertices is such a family, where every pair commutes and there exists a threshold graph whose Laplacian eigenvalues are simple.

\begin{lemma} \label{lemma_commuting_laplacians}
Let $G_1$ and $G_2$ be threshold graphs of the same order, not necessarily connected, and label the vertices in order according to the construction codes of the graphs. Let $L_1$ and $L_2$ be the Laplacian matrices of $G_1$ and $G_2$, respectively. Then $L_1L_2 = L_2L_1$.
\end{lemma}
\begin{proof}
We proceed by induction on $n$, the order of the graphs. The result clearly holds for $n=1$. For the inductive step, assume the lemma holds for graphs of order $n - 1$. We will show that it holds for graphs of order $n$. 

Let $G_1$ and $G_2$ be threshold graphs of order $n$, with construction sequences $b_1b_2\cdots b_{n-1} b_n$ and $c_1c_2\cdots c_{n-1}c_n$, respectively. Note that every threshold graph of order $n$ corresponds to a construction sequence of length $n$ that can be created by appending a 1 or 0 to the end of a construction sequence of length $n-1$. Let $L_1$ and $L_2$ be the Laplacian matrices of $G_1$ and $G_2$ respectively, and let $\hat{L}_1$ and $\hat{L}_2$ be the Laplacian matrices of the threshold graphs with construction codes $b_1b_2\cdots b_{n-1}$ and $c_1c_2\cdots c_{n-1}$, respectively. We will analyse the three possible cases: $b_n=c_n=1$, $b_n=c_n=0$, and $b_n\neq c_n$.
\begin{itemize}
\item[\emph{Case 1:}] $b_n=c_n=1$\\
Note that for $i=1, 2$:
\[L_i = \left[\begin{array}{c|c} \hat{L}_i + I & -\mathbbm{1}\\\hline -\mathbbm{1}^\top & n-1\end{array} \right],\]
where $\mathbbm{1}$ is a vector of all-ones and $I$ is the identity matrix of order $n-1$. Then 
\begin{eqnarray*}
L_1L_2 & = & \left[\begin{array}{c|c} \hat{L}_1 + I & -\mathbbm{1}\\\hline -\mathbbm{1}^\top & n-1\end{array} \right]\left[\begin{array}{c|c} \hat{L}_2 + I & -\mathbbm{1}\\\hline -\mathbbm{1}^\top & n-1\end{array} \right] \\
& = & \left[\begin{array}{c|c} (\hat{L}_1 + I)(\hat{L}_2+I) + \mathbbm{1}\mathbbm{1}^\top & -\hat{L}_1\mathbbm{1}-\mathbbm{1} - (n-1)\mathbbm{1}\\\hline -\mathbbm{1}^\top\hat{L}_2-\mathbbm{1}^\top - (n-1)\mathbbm{1}^\top & (n-1) + (n-1)^2\end{array} \right] \\
& = & \left[\begin{array}{c|c} \hat{L}_1\hat{L}_2 + \hat{L}_1 + \hat{L}_2 + I + \mathbbm{1}\mathbbm{1}^\top & -n\mathbbm{1}\\\hline -n\mathbbm{1}^\top & n(n-1)\end{array} \right],\\
& & \qquad \qquad \qquad \mbox{(since rows and columns of $\hat{L}_i$ sum to 0)} \\
& = & \left[\begin{array}{c|c} \hat{L}_2\hat{L}_1 + \hat{L}_1 + \hat{L}_2 + I + \mathbbm{1}\mathbbm{1}^\top & -n\mathbbm{1}\\\hline -n\mathbbm{1}^\top & n(n-1)\end{array} \right]\\
& & \qquad \qquad \qquad \mbox{(by the induction hypothesis)}\\
& = & L_2L_1.
\end{eqnarray*}
\item[\emph{Case 2:}] $b_n=c_n=0$\\
Note that 
\[L_1 = \left[\begin{array}{c|c} \hat{L}_1 & \mathbf{0} \\ \hline \mathbf{0}^\top & 0 \end{array}\right] \quad \mbox{and} \quad L_2 = \left[\begin{array}{c|c} \hat{L}_2 & \mathbf{0} \\ \hline \mathbf{0}^\top & 0 \end{array}\right].\]
Clearly,
\[L_1L_2 = \left[\begin{array}{c|c} \hat{L}_1\hat{L}_2 & \mathbf{0} \\ \hline \mathbf{0}^\top & 0 \end{array}\right] = \left[\begin{array}{c|c} \hat{L}_2\hat{L}_1 & \mathbf{0} \\ \hline \mathbf{0}^\top & 0 \end{array}\right] = L_2L_1,\]
by the induction hypothesis.

\item[\emph{Case 3:}] $b_n\neq c_n$\\
Without loss of generality, suppose $b_n=1$ and $c_n=0$. Then 
\[L_1 = \left[\begin{array}{c|c} \hat{L}_1 + I & -\mathbbm{1}\\\hline -\mathbbm{1}^\top & n-1\end{array} \right] \quad \mbox{and} \quad L_2 = \left[\begin{array}{c|c} \hat{L}_2 & \mathbf{0} \\ \hline \mathbf{0}^\top & 0 \end{array}\right].\]
Hence 
\begin{align*}
L_1L_2  = &~ \left[\begin{array}{c|c} \hat{L}_1 + I & -\mathbbm{1}\\\hline -\mathbbm{1}^\top & n-1\end{array} \right]\left[\begin{array}{c|c} \hat{L}_2 & \mathbf{0} \\ \hline \mathbf{0}^\top & 0 \end{array}\right] \\
 = &~ \left[\begin{array}{c|c} \hat{L}_1\hat{L}_2 + \hat{L}_2 & -\mathbf{0}\\\hline -\mathbf{0}^\top & 0 \end{array} \right] \\
 = &~ \left[\begin{array}{c|c} \hat{L}_2\hat{L}_1 + \hat{L}_2 & -\mathbf{0}\\\hline -\mathbf{0}^\top & 0 \end{array} \right]\\
 = &~~ L_2L_1.\qedhere
\end{align*}
\end{itemize}
\end{proof}

\begin{remark}
    In general, if written with respect to arbitrary orderings of the vertices, Laplacian matrices of threshold graphs do not necessarily commute. For instance, one can verify that the following matrices, both Laplacians of the path on three vertices, do not commute:
    \begin{align*}
         \begin{bmatrix}
            1 & 0 & -1\\0 & 1 & -1\\-1 & -1 & 2
        \end{bmatrix}\quad \text{and} \quad\begin{bmatrix}
            1 & -1 & 0\\-1 & 2 & -1\\0 & -1 & 1
        \end{bmatrix}.
    \end{align*}
    While the former is written with respect to the construction code $001$, the latter is not.
\end{remark}

Now we consider a threshold graph which has simple eigenvalues and find the corresponding eigenvectors, thus producing the unitary matrix $U$ which diagonalizes all other Laplacian matrices of threshold graphs written in the appropriate order. This threshold graph is the one with construction code $0101\cdots 01$ (for $n$ even) or $00101\cdots 01$ (for $n$ odd).

\begin{figure}
\begin{center}
\begin{tikzpicture}[every node/.style = {inner sep = 2pt}, scale =0.8]
\foreach \x in {1,...,10} \node (\x) at ({2*cos(18+36*(\x-1))}, {2*sin(18+36*(\x-1))}) {};
\draw[thick] (2)--(1);
\foreach \x in {1, 2, 3} \draw[thick] (4) -- (\x);
\foreach \x in {1, 2, 3, 4, 5} \draw[thick] (6) -- (\x);
\foreach \x in {1, 2, 3, 4, 5, 6, 7} \draw[thick] (8) -- (\x);
\foreach \x in {1, 2, 3, 4, 5, 6, 7, 8, 9} \draw[thick] (10) -- (\x);

\foreach \x in {1, ..., 10} \draw[black, fill=white] (\x) circle(5pt);
\end{tikzpicture}
\end{center}
\caption{The threshold graph on $n=10$ vertices with construction code $01010101$.}
\end{figure}
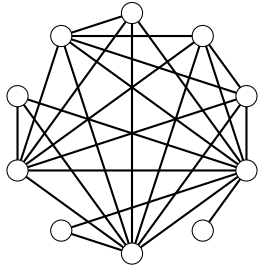


\begin{proposition}\label{prop:U} Let $n\geq 2$, and let $U$ be an upper Hessenberg unitary $n\times n$ matrix defined entrywise as follows:
	\begin{equation}\label{eq:U}
		U_{i,j} = \left\{\begin{array}{ll} \frac{-j}{\sqrt{j(j+1)}}, & \mbox{if } i = j+1;\\
			\frac{1}{\sqrt{j(j+1)}}, & \mbox{if } i\leq j < n;\\
			\frac{1}{\sqrt{n}}, & \mbox{if } j=n;\\
			0, & \mbox{if } i\geq j+2;\end{array}\right.
	\end{equation}
	Then $U$ diagonalizes the Laplacian matrix of any threshold graph of order $n$ written with respect to its construction code, and it is unique up to permutation of columns and signs of columns.
	\end{proposition}

\begin{proof}
	Consider the Laplacian matrix of the threshold graph with construction code $0101\cdots01$ in the case that $n$ is even, and $00101\cdots01$ in the case that $n$ is odd. By Theorem 5.3 in \cite{hammer1996laplacian}, the eigenvalues of $L$ are $0, 1,\dots,\lceil\frac{n}{2}\rceil-1,\lceil\frac{n}{2}\rceil+1,\dots,n$. Hence, all eigenvalues of $L$ are simple. From Theorem~\ref{thm:universal and unique} and Lemma~\ref{lemma_commuting_laplacians}, we only need to show that the columns of $U$ consists of normalized eigenvectors. 
	
	Let $n$ be even. Then the Laplacian matrix $L$ is given as follows:
	\[
	L = 
	\begin{bmatrix} 
		k & -1 & 0 & -1 & 0 & \cdots & -1 & 0 & -1 \\
		-1 & k & 0 & -1 & 0 & \cdots & -1 & 0 & -1 \\
		0 & 0 & k-1 & -1 & 0 & \cdots & -1 & 0 & -1 \\
		-1 & -1 & -1 & k+1 & 0 & \cdots & -1 & 0 & -1 \\
		0 & 0 & 0 & 0 & k-2 & \cdots & -1 & 0 & -1 \\
		\vdots & \vdots & \vdots & \vdots & \vdots & \ddots & \vdots & \vdots & \vdots \\
		-1 & -1 & -1 & -1 & -1 & \cdots & 2k-2 & 0 & -1 \\
		0 & 0 & 0 & 0 & 0 & \cdots & 0 & 1 & -1 \\
		-1 & -1 & -1 & -1 & -1 & \cdots & -1 & -1 & 2k-1 \\
	\end{bmatrix} 
	.\]
	One can verify that the following is an orthogonal set of eigenvectors of $L$:
\[\bv_1 = \begin{bmatrix} 1 \\ -1 \\ 0 \\ 0 \\ \vdots \\ 0 \end{bmatrix}, \bv_2 = \begin{bmatrix} 1 \\ 1 \\ -2 \\ 0 \\ \vdots \\ 0 \end{bmatrix}, \bv_3 = \begin{bmatrix} 1 \\ 1 \\ 1 \\ - 3\\ \vdots \\ 0 \end{bmatrix}, \ldots,  \bv_{n-1} = \begin{bmatrix} 1 \\ 1 \\ 1 \\ 1\\ \vdots \\ -(n-1) \end{bmatrix}, \bv_n = \begin{bmatrix} 1 \\ 1 \\ 1 \\ 1\\ \vdots \\ 1 \end{bmatrix}.\]
In particular, for $1\leq \lambda \leq k-1$, the corresponding eigenvector is 
\[\bv^{(\lambda)} = \begin{bmatrix} 1 & \cdots & 1 & -(n-2\lambda) & 0 & \cdots & 0 \end{bmatrix}^\top, \]
where there are $2\lambda-1$ zeros. For $k+1 \leq \lambda \leq 2k$, the corresponding eigenvector is 
\[\bv^{(\lambda)} = \begin{bmatrix} 1 & \cdots & 1 & (n-2\lambda+1) & 0 & \cdots & 0 \end{bmatrix}^\top,\]
where there are $2n-2\lambda$ zeros. 
	
	Similarly, it can be checked that the same vectors are also eigenvectors of the Laplacian matrix for the code $00101\cdots01$ when $n$ is odd. Therefore, normalizing each vector, our desired result is obtained. 
\end{proof}

\subsection{Finding $\mathcal{K}(G)$ using the unitary matrix $U$}
\label{subsec:K-expU}

In this section, we use the results obtained so far to compute Kemeny's constant for a connected threshold graph by leveraging the expression given in Proposition~\ref{prop:L-exp}. In particular, we note that we can compute $L^\dagger$ more easily via its diagonalized form. We have $L=U\Lambda U^\top$, where $U$ is given by \eqref{eq:U} and $\Lambda$ is a diagonal matrix consisting of the eigenvalues of $L$ in some order dictated by the order of the columns in $U$. Since $L$ has exactly one zero eigenvalue corresponding to the last column of $U$, note then that $L^\dagger = U\Lambda^\dagger U^\top$, where $\Lambda^\dagger = \diag(\frac{1}{\lambda_1}, \frac{1}{\lambda_2}, \ldots, \frac{1}{\lambda_{n-1}}, 0).$

\begin{proposition}\label{prop:KG_termsofU}
Let $G$ be a threshold graph with $n$ vertices and $m$ edges, and let $d_1, d_2, \ldots, d_n$ be the degrees of the vertices. Let $\lambda_1, \lambda_2, \ldots, \lambda_{n-1}, \lambda_n=0$ be the eigenvalues of the Laplacian matrix of $G$, indexed according to the order of the columns in the unitary matrix $U$ in \eqref{eq:U}. Then 
\[\mathcal{K}(G) = \frac{1}{2m}\sum_{i=1}^{n-1} \frac{1}{\lambda_i} \sum_{j<k} d_jd_k (U_{j,i} - U_{k, i})^2.\]
\end{proposition}

\begin{proof}
From Proposition~\ref{prop:L-exp}, we have
\[\mathcal{K}(G) = \mathbf{q}^\top \mathbf{d} - \frac{\mathbf{d}^\top L^\dagger \mathbf{d}}{2m},\]
where $\mathbf{d} = \begin{bmatrix} d_1 & d_2 & \cdots & d_n \end{bmatrix}$, and $\mathbf{q}$ is the vector of diagonal entries of $L^\dagger$.
We have 
\begin{eqnarray*}
\ell^\dagger_{k,k} = \be_k^\top U \Lambda^\dagger U^\top \be_k = \sum_{i=1}^{n-1} \frac{1}{\lambda_i} U_{k,i}^2.
\end{eqnarray*}
Hence 
\[\mathbf{q}^\top \mathbf{d} = \sum_{i=1}^{n-1} \frac{1}{\lambda}_i \sum_{j=1}^n d_j U_{j,i}^2.\]
Furthermore,
\begin{eqnarray*}
\mathbf{d}^\top L^\dagger \mathbf{d} = \mathbf{d}^\top U \Lambda^\dagger U^\top \mathbf{d} =  (U^\top \mathbf{d})^\top \Lambda^\dagger (U^\top \mathbf{d}) = \sum_{i=1}^{n-1} \frac{1}{\lambda_i} (\sum_{j=1}^n d_j U_{j,i})^2.  
\end{eqnarray*}
Combining these, we have 
\begin{equation}\label{eq:KG_use}
\mathcal{K}(G) = \sum_{i=1}^{n-1} \frac{1}{\lambda_i}\left(\sum_{j=1}^n d_jU_{j,i}^2 - \frac{(\sum_{j=1}^n d_jU_{j,i})^2}{2m}\right).\end{equation}
Noting that $2m=\sum_{k=1}^n d_k$ allows us to simplify as follows:
\begin{eqnarray*}
\mathcal{K}(G) & = & \frac{1}{2m}\sum_{i=1}^{n-1} \frac{1}{\lambda_i}\left[(\sum_{k=1}^n d_k)(\sum_{j=1}^n d_j U_{j,i}^2) - (\sum_{j=1}^n d_jU_{j,i})^2\right] \\
 & = & \frac{1}{2m}\sum_{i=1}^{n-1}\frac{1}{\lambda_i}\left[\sum_{j=1}^n d_j^2U_{j,i}^2 + \sum_{j=1}^n\sum_{\substack{k=1\\k\neq j}}^n d_jd_kU_{j,i}^2 - \sum_{j=1}^n d_j^2U_{j,i}^2 - \sum_{j=1}^n\sum_{\substack{k=1\\k\neq j}}^n d_jd_kU_{j,i}U_{k,i}\right].
\end{eqnarray*}
This may be simplified to obtain
\begin{equation*} 
\mathcal{K}(G) = \frac{1}{2m}\sum_{i=1}^{n-1} \frac{1}{\lambda_i}\sum_{j< k} d_jd_k(U_{j,i}-U_{k,i})^2.\qedhere
\end{equation*}
\end{proof}

The expression for $\mathcal{K}(G)$ may be improved by computing each $\lambda_i$ explicitly from the code. We note again that the order of the eigenvalues of the Laplacian matrix in this expression is determined by the order of the columns in $U$, given in Proposition~\ref{prop:U}. Thus $\lambda_{n}=0$ for every threshold graph of order $n$; to determine $\lambda_i$, we write the Laplacian matrix with respect to the vertex ordering given by the construction code of the graph, and compute $\mathbf{u}_i^\top L\mathbf{u}_i$, where $\mathbf{u}_i$ is the $i^{th}$ column of $U$ in \eqref{eq:U}. We will denote by $\theta_i$ the number of $1$'s in the code after position $i$, and by $c_i$ the binary indicator of the code at position $i$ ($c_i=1$ if it corresponds to a $1$ and $c_i=0$ if it corresponds to a $0$).

\begin{remark}
We note that the eigenvalues of the Laplacian matrix of a threshold graph are known; see \cite{hammer1996laplacian}. The purpose of the next result is to facilitate the application of the result from Wang et al (see \cite{wang2017kemeny}), and in order to do so, we focus on finding the eigenpairs of the Laplacian matrix in order, with respect to the ordering of the eigenvectors given in Proposition~\ref{prop:U}.
\end{remark}

\begin{lemma} \label{eq_lambda}
Let $G$ be a threshold graph with construction code $c_1c_2\cdots c_n$. Let $L$ be the Laplacian matrix written with the vertices ordered according to the construction code. Then $U LU^\top = \Lambda$, where $U$ is given by \eqref{eq:U}, and the $i^{th}$ diagonal entry of $\Lambda$ is given by:
\begin{equation*} 
\lambda_{i} = \theta_i + ic_{i+1}.
\end{equation*}
\end{lemma}

\begin{proof}
The $i^{th}$ eigenvector according to the ordering in Proposition~\ref{prop:U} is 
\[\mathbf{v}_i = \begin{bmatrix} \mathbbm{1} \\\hline -i\mathbf{e}_1\end{bmatrix},\]
where $\mathbf{e}_1$ is the first standard basis vector in $\mathbb{R}^{n-i}$. Note that $\lambda_i = \frac{\bv_i^\top L \bv_i}{\|\bv_i\|^2}$, and 
\[\|\bv_i\|^2 = i + i^2 = i(i+1).\]

Partitioning $L$ conformally with $\bv_i$, we obtain
\[\bv_i^\top L \bv_i = \left[\begin{array}{c|c} \mathbbm{1}^\top & -i\be_1^\top \end{array}\right]\left[\begin{array}{c|c} \hat{L} + \theta_i I & L_{12} \\\hline L_{12}^\top & L_{22}\end{array}\right]\begin{bmatrix} \mathbbm{1} \\\hline -i\mathbf{e}_1\end{bmatrix}, \]
where $\hat{L}$ is the Laplacian matrix of the threshold graph with construction code $c_1c_2\cdots c_i$.
Multiplying, we get
\[
\bv_i^\top L \bv_i  =  \mathbbm{1}^\top \hat{L}\mathbbm{1} + \theta_i\mathbbm{1}^\top\mathbbm{1} - i\mathbbm{1}^\top L_{12}\be_1 - i\be_1^\top L_{12}^\top\mathbbm{1} + i^2\be_1^\top L_{22}\be_1.
\]
Observe the following: 
\begin{itemize}
\item $\mathbbm{1}^\top\hat{L}\mathbbm{1} = 0$
\item $L_{12}\be_1$ is dependent on the value of the construction code at the corresponding vertex; in particular, this is $-\mathbbm{1}$ if $c_{i+1}=1$, and $\mathbf{0}$ if $c_{i+1} = 0$.
\item $\be_1^\top L_{22} \be_1$ is the degree of the $(i+1)^{th}$ vertex in the construction code, and is $\theta_{i+1} + i$ if $c_{i+1}=1$, and $\theta_i$ if $c_{i+1} = 0$.
\item Finally, if $c_{i+1} = 1$, then $\theta_{i+1} = \theta_{i} - 1$, and if $c_{i+1} = 0$, then $\theta_{i+1}=\theta_{i}.$
\end{itemize}
Hence if $c_{i+1} = 1,$ we have
\begin{eqnarray*}
\bv_i^\top L \bv_i = i\theta_{i} + 2i^2 + i^2\theta_{i+1} + i^3 = \theta_{i}[i(i+1)] + i^2(i+1),
\end{eqnarray*}
and if $c_{n+2-i} = 0$, we have
\begin{eqnarray*}
\bv_i^\top L \bv_i = i\theta_i + i^2\theta_{i+1} =  i\theta_i + i^2\theta_{i} = \theta_{i}[i(i+1)].
\end{eqnarray*}
Hence 
\[\lambda_i = \frac{\bv_i^\top L \bv_i}{\|\bv_i\|^2} = \left\{\begin{array}{lc} \theta_i + i, & \mbox{if } c_{i+1} = 1,\\
\theta_{i}, & \mbox{if } c_{i+1} = 0.\end{array}\right.\qedhere\]
\end{proof}

Our goal is now to express Kemeny's constant entirely in terms of the construction code for the threshold graph, using the expression in Proposition~\ref{prop:KG_termsofU} and Lemma~\ref{eq_lambda}.

\begin{theorem} \label{nicer_kemenys_in_terms_of_code}
Let $G$ be a connected threshold graph with $n$ vertices and $m$ edges, and let $\mathbf{c}^\top = \begin{bmatrix} c_1 & c_2 & \cdots & c_n\end{bmatrix}$ be the construction code of $G$, written as a vector in $\mathbb{R}^n$. Define \begin{eqnarray}\label{eq:vectors}
\mathbf{w}_k^\top & = & \begin{bmatrix} 0 & 2 & 4 & \cdots & 2(k-1) & -k(k-1)& 0 & \cdots & 0 \end{bmatrix} \\
\mathbf{z}_k^\top & =  & \begin{bmatrix} 0 & 0 & 0 & \cdots & 0 & (k+1) & 1 & \cdots & 1 \end{bmatrix}, \notag
\end{eqnarray}
where the $k+1$ appears in the $(k+1)$ position of $\bz_k$.
Then we have
\begin{equation}\label{expression:kem} 
\mathcal{K}(G) = n-1 -\sum_{i=1}^{n-1}\frac{c_{i+1}}{\bz_i^\top \bc} + \sum_{i=1}^{n-1} \frac{(\bw_i^\top\bc)(2m-\bw_i^\top\bc)}{2m\cdot i(i+1) \cdot \bz_i^\top\bc}.
\end{equation}
\end{theorem}
\begin{proof}
We express $\mathcal{K}(G)$ in terms of the code $c$ by finding appropriate expressions for each part of the expression in \eqref{eq:KG_use}, rewritten here:
\[\mathcal{K}(G) = \sum_{i=1}^{n-1} \frac{1}{\lambda_i}\left(\sum_{j=1}^n d_jU_{j,i}^2 - \frac{1}{2m}\left(\sum_{j=1}^n d_jU_{j,i}\right)^2\right).\]
We first note the following:
\begin{itemize}
\item Since $\theta_i$ represents the number of ones in $\bc$ after position $i$, substituting this for $\lambda_i = \theta_i + ic_{i+1}$ gives
\begin{equation}
\lambda_i = (i+1)c_{i+1} + c_{i+2} + c_{i+3} + \cdots + c_n.
\end{equation}
We write this as an inner product of $\bz_i$ with $\bc$; that is, $\lambda_i = \bz_i^\top\bc$.
\item The degree of the $i^{th}$ vertex may be computed as
\begin{equation} \label{eq_degree}
d_i = (i-1)c_i + \theta_i = (i-1)c_i + c_{i+1} + c_{i+2} + \cdots + c_n 
\end{equation}
which can be written as an inner product $(\bz_{i-1}-\be_i)^\top\bc$.
\item The partial sum $\sum_{j=1}^i d_j$ is given by
\begin{eqnarray*}
\sum_{j=1}^i d_j & = & \sum_{j=1}^i\left((j-1)c_j +\sum_{r=j+1}^n c_r\right) \\
& = & c_2 + 2c_3 + 3c_4 + \cdots + (i-1)c_i + (c_2 + c_3+\cdots+c_n) + (c_3+c_4 + \cdots + c_n) \\
& & \qquad \qquad + \cdots + (c_{i+1} + \cdots + c_n) \\
& = & 2c_2 + 4c_3 + 6c_4 + \cdots + 2(i-1)c_i + ic_{i+1} + ic_{i+2} + \cdots + ic_n \\
& = & 2c_2 + 4c_3 + 6c_4 + \cdots + 2(i-1)c_i + 2ic_{i+1} - ic_{i+1} + ic_{i+2} + \cdots + ic_n.
\end{eqnarray*}
Define $\hat{\bw}_k^\top = \begin{bmatrix} 0 & 2 & 4 & \cdots & 2k & 0 & \cdots & 0 \end{bmatrix}$.
Then we can write this as \[\hat{\bw}_i^\top \bc + i(\bz_i - (i+2)\be_{i+1})^\top\bc.\]

\end{itemize}
We now compute the remaining components of the expression in \eqref{eq:KG_use}. We have
\begin{eqnarray*}
\sum_{j=1}^n d_j U_{j,i}^2 & = & \frac{1}{i(i+1)}\left(\sum_{j=1}^i d_j + i^2 d_{i+1}\right) \\ 
& = & \frac{1}{i(i+1)}\left(\hat{\bw}_i^\top \bc + i(\bz_i - (i+2)\be_{i+1})^\top\bc + i^2(\bz_i - \be_{i+1})^\top \bc \right) \\
& = & \frac{1}{i(i+1)}\left(\hat{\bw}_i^\top \bc + i\bz_i^\top\bc - i(i+2)c_{i+1} + i^2\bz_i^\top\bc - i^2c_{i+1} \right)\\
& = & \frac{1}{i(i+1)}\hat{\bw}_i^\top\bc + \bz_i^\top\bc - 2c_{i+1}.
\end{eqnarray*}
We also have 
\begin{eqnarray*}
\left(\sum_{j=1}^n d_j U_{j,i}\right)^2 & = & \frac{1}{i(i+1)}\left(\sum_{j=1}^i d_j -id_{i+1}\right)^2\\
& = & \frac{1}{i(i+1)}\left(\hat{\bw}_{i}^\top\bc + i\bz_i^\top\bc -i(i+2)c_{i+1} -i(\bz_i^\top-\be_{i+1}^\top)\bc \right)^2\\
& = &\frac{1}{i(i+1)}\left(\hat{\bw}_i^\top\bc -i(i+1)c_{i+1}\right)^2.
\end{eqnarray*}
Hence 
\begin{eqnarray*}
\mathcal{K}(G) & = & \sum_{i=1}^{n-1}\frac{1}{\bz_i^\top\bc}\left(\frac{\hat{\bw}_i^\top\bc}{i(i+1)} + \bz_i^\top\bc -2c_{i+1} -\frac{(\hat{\bw}_i^\top\bc -i(i+1)c_{i+1})^2}{2m\cdot i(i+1)}\right)\\
& = & n-1 -\sum_{i=1}^{n-1} \frac{2c_{i+1}}{\bz_i^\top\bc} + \sum_{i=1}^{n-1} \frac{2m\hat{\bw}_i^\top\bc - (\hat{\bw}_i^\top\bc -i(i+1)c_{i+1})^2}{2m\cdot i(i+1) \cdot \bz_i^\top\bc}.
\end{eqnarray*}
By making the substitution $\bw_i := \hat{\bw}_i - i(i+1)\be_{i+1}$ in the last sum and simplifying, we obtain 
\[\mathcal{K}(G) = (n-1) -\sum_{i=1}^{n-1}\frac{c_{i+1}}{\bz_i^\top \bc} + \sum_{i=1}^{n-1} \frac{(\bw_i^\top\bc)(2m-\bw_i^\top\bc)}{2m\cdot i(i+1) \cdot \bz_i^\top\bc}.\qedhere\]
\end{proof}

Alternatively, we can note that $2m = \sum_{j=1}^n d_i$, and that $\lambda_{i}=\theta_i + ic_{i+1} = d_{i+1} + c_{i+1}$. Now, plugging these, in conjunction with the result in Proposition~$\ref{prop:U}$, into the formula in Proposition~$\ref{prop:KG_termsofU}$, we obtain the following result:

\begin{corollary}\label{eq_Kemenys_mostly_degree}
Let $G$ be a threshold graph with $n$ vertices. Then we have:
\begin{equation*} 
\mathcal{K}(G) = \sum_{j=2}^{n} \frac{1}{d_j+c_j}\cdot \frac{1}{j(j-1)}\left[\sum_{i=1}^{j-1} d_i + (j-1)^2 d_j - \frac{(\sum_{i=1}^{j-1} d_i - (j-1)d_j)^2}{\sum_{i=1}^n d_i}\right].
\end{equation*}
\end{corollary} 




From \eqref{expression:kem}, we shall obtain upper bounds on Kemeny's constant of a threshold graph. 

\begin{proposition}
Let $G$ be a connected threshold graph with $n$ vertices, where $n\geq 3$. Then $$\mathcal{K}(G) < 2n - 3.$$    
\end{proposition}
\begin{proof}
Consider $\mathcal{K}(G)$ in \eqref{expression:kem}:
\[\mathcal{K}(G) = (n-1) -\sum_{i=1}^{n-1}\frac{c_{i+1}}{\bz_i^\top \bc} + \sum_{i=1}^{n-1} \frac{(\bw_i^\top\bc)(2m-\bw_i^\top\bc)}{2m\cdot i(i+1) \cdot \bz_i^\top\bc}.\]

We claim that $\frac{(\bw_i^\top\bc)(2m-\bw_i^\top\bc)}{2m\cdot i(i+1) \cdot \bz_i^\top\bc}\leq \frac{1}{\mathbf{z}^\top \bc}$ for $1\leq i\leq n-2$. If $\bw_i^\top\bc\leq 0$ then $\frac{(\bw_i^\top\bc)(2m-\bw_i^\top\bc)}{2m\cdot i(i+1) \cdot \bz_i^\top\bc}\leq 0<\frac{1}{\bz_i^\top\bc}$. Suppose that $\bw_i^\top\bc\geq 1$. Note that $\mathbf{w}_i^\top\bc\leq 2(1+\cdots+i-1) =  i(i-1)$. Since $2m-\mathbf{w}_i^\top\bc<2m$, we have
$$\frac{(\bw_i^\top\bc)(2m-\bw_i^\top\bc)}{2m\cdot i(i+1) \cdot \bz_i^\top\bc}<\frac{i(i-1)}{i(i+1) \cdot \bz_i^\top\bc}<\frac{1}{\bz_i^\top\bc}.$$
We note that $\bw_{n-1}^\top\bc<0$. Therefore, $\mathcal{K}(G)< n - 1 + \sum_{i=1}^{n-2} \frac{1}{\mathbf{z}_i^\top \bc} < n - 1 + n - 2 = 2n - 3$.
\end{proof}

\begin{remark}
    It appears in \cite{kim2023bounds} that Kemeny's constant of a graph $G$ on $n$ vertices with maximum degree $n-1$ is bounded above by $2n-2$.
\end{remark}



We now provide an upper bound which is more sharp when a threshold graph is sparse. 

\begin{proposition}
Let $G$ be a connected threshold graph with $n$ vertices with $m$ edges, where $n\geq 3$. Then $$\mathcal{K}(G) < n -1 + \frac{3}{2}\sqrt{m}.$$
\end{proposition}
\begin{proof}
Consider $\mathcal{K}(G)$ in \eqref{expression:kem}:
\[\mathcal{K}(G) = (n-1) -\sum_{i=1}^{n-1}\frac{c_{i+1}}{\bz_i^\top \bc} + \sum_{i=1}^{n-1} \frac{(\bw_i^\top\bc)(2m-\bw_i^\top\bc)}{2m\cdot i(i+1) \cdot \bz_i^\top\bc}.\]
Note that $\mathbf{w}_j^\top \mathbf{c} \leq 0 + 2 + 4 + \ldots + 2(j - 1) = j(j - 1)$ and $\mathbf{z}_j^\top\mathbf{c}\geq 1$. We observe the following:
\begin{itemize}
    \item By the AM-GM inequality, we have $(\mathbf{w}_j^\top \mathbf{c})(2m-\mathbf{w}_j^\top \mathbf{c})\leq m^2$. Hence 
    $$\frac{(\bw_j^\top\bc)(2m-\bw_j^\top\bc)}{2m\cdot j(j+1) \cdot \bz_i^\top\bc}\leq \frac{m}{2j(j+1)(\mathbf{z}_j^\top \mathbf{c})}\leq \frac{m}{2j(j+1)}.$$
    \item If $j(j-1) \leq m$, then $(\mathbf{w}_j^\top \mathbf{c})(2m-\mathbf{w}_j^\top \mathbf{c})\leq j(j-1)(2m - j(j-1))$. Thus
    \begin{align}\label{temp:eqn}
        \frac{(\bw_j^\top\bc)(2m-\bw_j^\top\bc)}{2m\cdot j(j+1) \cdot \bz_j^\top\bc}\leq \frac{2m - j(j-1)}{2m\cdot \bz_j^\top\bc}<1.
    \end{align}
\end{itemize}

Note that $j(j-1) \leq m$ if and only if $1\leq j\leq \frac{1}{2}+\sqrt{\frac{1}{4}+m}$. Since $\sqrt{m}<\frac{1}{2}+\sqrt{\frac{1}{4}+m}$, \eqref{temp:eqn} holds for $1\leq j\leq \lfloor\sqrt{m}\rfloor$. Now we see
\begin{align*}
    \mathcal{K}(G) <& ~n-1 + \sum_{i=1}^{n-1} \frac{(\bw_i^\top\bc)(2m-\bw_i^\top\bc)}{2m\cdot i(i+1) \cdot \bz_i^\top\bc}\\
    < &~ n-1+\lfloor\sqrt{m}\rfloor+\frac{m}{2}\sum_{i=\lfloor\sqrt{m}\rfloor+1}^{n-1}\frac{1}{i(i+1)}\\
    = &~ n-1+\lfloor\sqrt{m}\rfloor+\frac{m}{2}\left(\frac{1}{\lfloor\sqrt{m}\rfloor+1}-\frac{1}{n}\right)\\
    <&~ n-1+\sqrt{m}+\frac{\sqrt{m}}{2}.\qedhere
\end{align*}


\end{proof}

\subsection{Example: Maximizing $\mathcal{K}(G)$ for threshold graphs}
A goal of this investigation is to determine the range of values of $\mathcal{K}(G)$ for threshold graphs. While it is clear that the graph that minimizes $\mathcal{K}(G)$ is the complete graph (which is itself a threshold graph), the threshold graph maximizing $\mathcal{K}(G)$ is not clear. Determining this would shed light on more general extremal questions regarding $\mathcal{K}(G)$. 

The results in the previous sections provide us several advantages in pursuing this extremal question. Typically, calculation of Kemeny's constant for a graph of order $n$ requires the computation of the eigenvalues of $D^{-1}A$, and thus has computational complexity on the order of $O(n^3)$. An advantage to the formula given above is a massive reduction in complexity, such that $\mathcal{K}(G)$ can be computed for a threshold graph of order $n$ in $\mathcal{O}(n^2)$ steps. This allows us to run a comprehensive search for all threshold graphs to determine the construction code for the threshold graph of order $n$ with maximum $\mathcal{K}(G)$. In particular, we have the following conjecture, which has been confirmed by computer for threshold graphs of all orders up to $n=30$.

\begin{conjecture}The threshold graph of order $n$ with maximum Kemeny's constant has construction code of the form 
0111..1100000...0001.
\end{conjecture}

This graph is obtained by adding a dominant vertex to a disjoint union $K_{r+1}\dot\cup \overline{K}_{n-r-2}$ of a complete graph on $r+1$ vertices and an empty graph on $n-r-2$ vertices, and is sometimes referred to as a pineapple graph. These have arisen before as extremal graphs when maximizing various measures of irregularity \cite{taittobin}.

In this section, we demonstrate the use of the formula in Proposition~\ref{nicer_kemenys_in_terms_of_code} by applying it to the family of pineapple graphs. We use the notation $0\mathbf{1}^r \mathbf{0}^{n-r-2}1$ to represent the construction code of such a graph, in which $\mathbf{1}^r$ represents a `block' of $r$ ones, and $\mathbf{0}^{n-r-2}$ represents a `block' of $n-r-2$ zeros.

\begin{lemma}\label{lem:pineapple}
Let $G$ be a threshold graph of order $n$ with construction code $0\mathbf{1}^r \mathbf{0}^{n-r-2}1$. Then 
\begin{align}\label{Kem:pineapple}
    \mathcal{K}(G) = n-4 +\frac{2}{r+2} + \frac{(n-1)(2r+3)}{2n+r^2 + r-2}.
\end{align}
\end{lemma}
\begin{proof}
The construction code of $G$ may be written as a vector 
\[\bc^\top = \begin{bmatrix} 0 & 1 & \cdots & 1 & 0 & \ldots & 0 & 1 \end{bmatrix}.\]
It is easily seen that
\[ \bz_i^\top\bc = \left\{\begin{array}{ll} r+2, & 1\leq i\leq r;\\
1, & r+1 \leq i\leq n-2;\\
n, & i = n-1.\end{array}\right.\]
Furthermore, 
\[\bw_i^\top\bc = \left\{\begin{array}{ll} 0, & 1\leq i\leq r;\\
r(r+1), & r+1 \leq i\leq n-2;\\
r(r+1)-(n-1)(n-2), & i = n-1.\end{array}\right.\]
Using these in \eqref{expression:kem}, we have that the first summation term is 
\[\sum_{i=1}^{n-1} \frac{c_{i+1}}{\bz_i^\top \bc} = \sum_{i=1}^{r} \frac{1}{r+2} + \frac{1}{n}= \frac{r}{r+2} +\frac{1}{n} = 1- \frac{2}{r+2} + \frac{1}{n}.\]
The second summation term is 
\begin{align*}
\sum_{i=1}^{n-1} \frac{(\bw_i^\top\bc)(2m-\bw_i^\top\bc)}{2mi(i+1)\bz_i^\top \bc}  = &~ \sum_{i=r+1}^{n-2} \frac{(r(r+1)(2m-r(r+1))}{2mi(i+1)}\\
 & +\frac{[r(r+1)-(n-1)(n-2)][2m-r(r+1)+(n-1)(n-2)]}{2mn^2(n-1)} \\
= &~ \frac{(r(r+1)(2m-r(r+1))}{2m}\sum_{i=r+1}^{n-2} \left(\frac{1}{i} - \frac{1}{i+1}\right)\\
& +\frac{[r(r+1)-(n-1)(n-2)][2m-r(r+1)+(n-1)(n-2)]}{2mn^2(n-1)} \\
 = & ~\frac{(r(r+1)(2m-r(r+1))}{2m} \left(\frac{1}{r+1} - \frac{1}{n-1}\right)\\
&+ \frac{[r(r+1)-(n-1)(n-2)][2m-r(r+1)+(n-1)(n-2)]}{2mn^2(n-1)}.
\end{align*}
Some tedious computation and simplification produces that the above is equal to $\frac{(n-1)(2r+3)}{2m}$, and substituting $2m = r(r+1) + 2(n-1)$ gives the final result.
\end{proof}


\begin{proposition}\label{prop:max_r_value}
    Let $n\geq 3$. Let $G$ be a threshold graph of order $n$ with construction code $0\mathbf{1}^r\mathbf{0}^{n-r-2}1$. Then $\mathcal{K}(G)$ is maximized when $r$ is one of the integers in $\{\lfloor \sqrt{2n} \rfloor - 1, \lfloor \sqrt{2n} \rfloor \}$ if $3\leq n\leq 20$, and when $r$ is one of the integers in $\{\lfloor \sqrt{2n} \rfloor, \lfloor \sqrt{2n} \rfloor +1\}$ if $n\geq 21$.
\end{proposition}
\begin{proof}
    Fix $n\geq 3$. Define $g:[0,\infty]\rightarrow\mathbb{R}$ by $g(r)$ being equal to the right side of \eqref{Kem:pineapple}. Let $f(r) = g(r-1)$. Then
    \begin{align*}
        \frac{df}{dr} = \frac{-2n r^4 + (10-6n) r^3 + (4n^2-21n+15) r^2 + (8n^2-8n) r + (-4n^2+9n-5)}{(r + 2)^2(r^2 + r + 2n - 2)^2}. 
    \end{align*}
    Let $P(r)$ denote the numerator of $\frac{df}{dr}$. There are two sign changes for the coefficients of the polynomial. So, by Descartes' rule, there are either $0$ or $2$ positive real roots of $\frac{df}{dr}$.
    (One can verify that the signs of the coefficients in $\frac{df}{dr}$ depends on values of $n$.) A basic computation gives that $f(0)=f(1)=n-\frac{3}{2}$. There exists a root between $0$ and $1$. Hence there are exactly two positive roots. 

    We claim that the other positive root is in $(\sqrt{2n}-2,\sqrt{2n}-1)$ for $3\leq n\leq 20$ and $(\sqrt{2n}-1,\sqrt{2n})$ for $n\geq 21$. Then one can show that for $n\geq 3$,
    \begin{align*}
        P(\sqrt{2n}-2) =&~ (12\sqrt{2n} - 70)n^2+ (88\sqrt{2n} - 133)n +60\sqrt{2n}  - 25>0, \\
        P(\sqrt{2n}) =&~- 4n^2\sqrt{2n} - 46n^2 + 12n\sqrt{2n} + 39n - 5<0.
    \end{align*}
    (It can be easily seen that the derivative of $P(\sqrt{2n}-2)$ is positive for $n\geq 3$.) Also, $P(\sqrt{2n}-1) = 4n^2\sqrt{2n} - 38n^2+44n\sqrt{2n} - 30n<0$ for $3\leq n\leq 20$. Therefore, our desired result is established.
\end{proof}

%




We refine our conjecture as follows:

\begin{conjecture}
Let $n\geq 3$. Then the maximum value of $\mathcal{K}(G)$ over all threshold graphs $G$ of order $n$ is $n + \frac{\sqrt{n}}{2} + \mathcal{O}(1)$, and is achieved for the graph with construction code $0\mathbf{1}^r\mathbf{0}^{n-r-2}1$, where $r$ is as in Prop.~\ref{prop:max_r_value}.
\end{conjecture}

\section{Accessibility index in threshold graph}
\label{The Moment and the Accessibility Index}
In this section, we shall completely determine the ordering of accessibility indices of vertices in a threshold graph. 

We recall that the moment of $v$ in a graph $G$ can be expressed as
\begin{equation*}
	\mu(v) = \sum_{j\neq v} d_jr_{j,v}. 
\end{equation*}
Since
\[
\mu(v) = \alpha(v) + \mathcal{K}(G),
\]
the ordering of vertices according to moment is the same as that to accessibility index. Hence, one could find moments explicitly and compare them. Indeed, an explicit expression of $r_{i,j}$ can be provided in terms of the construction code and degrees as follows. 


\begin{proposition}
    Let $G$ be a threshold graph with construction code $c_1c_2\cdots c_n$. Then,
    \begin{align*}
        r_{j,v} = \begin{cases}
        0, & \text{ if } j = v; \\
        \frac{1}{d_j+c_j}\frac{j - 1}{j} + \frac{1}{d_v+c_v}\frac{v}{v - 1} + \sum_{i = j}^{v-2} \frac{1}{d_{i+1}+c_{i+1}}\frac{1}{i(i+1)}, & \text{ if } j < v.  
        \end{cases}
    \end{align*}
\end{proposition}
\begin{proof}
    Recall that $r_{i,j}=\ell_{i,i}^\dagger + \ell_{j,j}^\dagger - 2\ell_{i,j}^\dagger$ where $L^\dagger$ is the Moore-Penrose inverse of the Laplacian matrix $L$ of $G$. By Proposition~\ref{prop:U}, $$r_{j,v} = \sum_{i=1}^{n-1} \frac{1}{\lambda_i}(U_{j,i}-U_{v,i})^2.$$ 
    We note that $\lambda_{i}=\theta_i + ic_{i+1} = d_{i+1} + c_{i+1}$. Using \eqref{eq:U}, the conclusion follows.
\end{proof}

It can be found in \cite[Theorem~5.4]{hammer1996laplacian} that the number of spanning trees of a threshold graph is expressed in terms of vertex degrees. Using the result, we obtain the following.

\begin{corollary}
    Let $G$ be a threshold graph with construction code $C = \mathbf{0}^{n_1}\mathbf{1}^{n_2}\ldots \mathbf{0}^{n_{2k-1}}\mathbf{1}^{n_{2k}}$. Let $v_\ell$ denote the degree of a vertex corresponding to $\ell^\text{th}$ block of $C$. We also let $k_1 = \lfloor(k-1)/2\rfloor$ and $k_2 = \lfloor(k-2)/2\rfloor$. Then,
    \begin{align*}
        f_{j,v} = \begin{cases}
        0, & \text{ if } j = v; \\
        \tau\left(\frac{1}{d_j+c_j}\frac{j - 1}{j} + \frac{1}{d_v+c_v}\frac{v}{v - 1} + \sum_{i = j}^{v-2} \frac{1}{d_{i+1}+c_{i+1}}\frac{1}{i(i+1)}\right), & \text{ if } j < v,
        \end{cases}
    \end{align*}
    where 
    \begin{align*}
        \tau = 
        \begin{cases}
            \prod_{\ell=1}^{k_1}v_i^{n_i}\prod_{i=k_1+2}^{2k-1}(v_i+1)^{n_i}(v_{k_1+1}+1)^{n_{k_1+1}-1}(v_{2k}+1)^{n_{2k}-1}, & \text{if $n_1>1$;}\\
            v_1^{n_1+n_2}\prod_{\ell=3}^{k_2}v_i^{n_i}\prod_{i=k_2+2}^{2k-1}(v_i+1)^{n_i}(v_{k_2+1}+1)^{n_{k_2+1}-1}(v_{2k}+1)^{n_{2k}-1}, & \text{if $n_1=1$.}
        \end{cases}
    \end{align*}
\end{corollary}

Employing the explicit expression of resistance distance would pose an intricacy in achieving the objective of this section. Instead, we shall adopt an alternative approach through enumerative combinatorics.

\subsection{Comparison of entries in the matrix $F$}
\label{Comparison of entries in the matrix $F$}

Let $G$ be a graph. We use $\cF(v;w)$ to denote the set of spanning $2$-forests of $G$ separating $v$ and $w$. Similarly, let $\cF(i,v;w)$ denote the set of spanning $2$-forests of $G$ such that one subtree contains $i$ and $v$, and the other contains $w$. We readily observe that for three vertices $x,y,z$, the set $\cF(x;y)$ can be expressed as a disjoint union as follows:
\begin{align}\label{eqn:disjoint}
	\cF(x;y) = \cF(z,x;y)\cup \cF(z,y;x).
\end{align}
For vertex $x\in V(G)$, let $\cF_x(i,v;w)$ be the subset of $\cF(i,v;w)$ such that for each forest, $x$ is adjacent to $v$, and $x$ lies on the path from $v$ to $i$.

\begin{lemma}\label{lem:bijection common nbh}
	Let $G$ be a graph. Then for distinct vertices $i,v,w$ of $G$, we have
	\begin{align*}
		\left|\bigcup_{z\in Z}\cF_z(i,w;v)\right| = \left|\bigcup_{z\in Z}\cF_z(i,v;w)\right|.
	\end{align*}
	where $Z$ is the set of common neighbours of $v$ and $w$.
\end{lemma}
\begin{proof}
	Let $z\in Z$. Define a map $\phi$ from  $\bigcup_{z\in Z}\cF_z(i,v;w)$ to $\bigcup_{z\in Z}\cF_z(i,w;v)$ as follows: for each forest $f\in\bigcup_{z\in Z}\cF_z(i,v;w)$, $\phi(f)$ is obtained from $f$ by removing the edge $\{z,v\}$ and inserting the edge $\{z,w\}$. Similarly, we define a map $\psi$ from $\bigcup_{z\in Z}\cF_z(i,w;v)$ to $\bigcup_{z\in Z}\cF_z(i,v;w)$ as follows: for each forest $f\in\bigcup_{z\in Z}\cF_z(i,w;v)$, $\psi(f)$ is obtained from $f$ by removing the edge $\{z,w\}$ and inserting the edge $\{z,v\}$. Since $z$ is adjacent to $v$ and $w$ both in $G$, $\phi$ and $\psi$ are well-defined. (We may have $\cF_z(i,w;v)=\emptyset$, and we can readily see from $\phi$ and $\psi$ that $\cF_z(i,w;v)\neq \emptyset$ if and only if $\cF_z(i,v;w)\neq \emptyset$.) Clearly, $\phi\circ\psi$ and $\psi\circ\phi$ are the identity maps. Hence, $\phi$ and $\psi$ are bijective. Since $\cF_{z_1}(i,v;w)$ and $\cF_{z_2}(i,v;w)$ are disjoint whenever $z_1\neq z_2$, our desired result is obtained.
\end{proof}

\begin{lemma}\label{lem: f(iv)<f(iw)}
	Let $G$ be a graph, and let $i,v,w$ be distinct vertices. If $N(w)\subseteq N(v)\cup\{v\}$, then
	\begin{align*}
		|\cF(i;v)| \leq |\cF(i;w)|.
	\end{align*}
	where the equality holds if and only if every path from $v$ to $i$ in $G$ satisfies one of the following: (i) none of $N(v)\backslash N(w)$ lies on the path; and (ii) if any vertex of $N(v)\backslash N(w)$ is adjacent to $v$ in the path, then $w$ must lie on the path.
\end{lemma}
\begin{proof}
	We note that $\cF(i,w;v) = \bigcup_{z\in N(w)}\cF_z(i,w;v)$ and $\cF(i,v;w) = \bigcup_{y\in N(v)}\cF_y(i,v;w)$. From \eqref{eqn:disjoint}, it suffices to show that $|\cF(i,w;v)|\leq |\cF(i,v;w)|$. We can see that
	\begin{align*}
		|\cF(i,w;v)| =&~ \left|\bigcup_{z\in N(w)}\cF_z(i,w;v)\right| \\
		=&~ \left|\bigcup_{z\in N(w)}\cF_z(i,v;w)\right| \hspace{15mm}\text{(by the hypothesis)}\\
		\leq &~ |\cF(i,v;w)|  \hspace{35mm}\text{(by Lemma~\ref{lem:bijection common nbh})}.
	\end{align*}
	
	Now we consider the condition for the equality. Suppose that there exists a path with a vertex $x\in N(v)\backslash N(w)$ so that it fails to satisfy both of (i) and (ii). Then $x$ is adjacent to $v$, and $w$ is not on the path. If $w$ is not a cut-vertex in $G$, then any spanning $2$-forest separating $w$ and the remaining belongs to $\cF(i,v;w)\backslash \cF(i,w;v)$. If $w$ is a cut-vertex, then there must be a component in the graph obtained from $G$ by deleting $w$ that contains the path from $v$ to $i$ (otherwise, $w$ would have lied on the path); so, any spanning $2$-forest separating the component and the remaining is in $\cF(i,v;w)\backslash \cF(i,w;v)$. Therefore, our desired result is established.
\end{proof}

\begin{theorem}\label{thm: entries f}
	Let $G$ be a threshold graph with code $C$. Then the following hold:
	\begin{itemize}
		\item[(i)] If $C = x_1\underbracket{0}_{\text{v}}\underbracket{0}_{\text{w}}x_2$ or $C = x_1\underbracket{1}_{\text{v}}\underbracket{1}_{\text{w}}x_2$, then $f_{i,v}=f_{i,w}$ for all $i\notin\{v,w\}$.
		\item[(ii)] If $C = x_1\underbracket{0}_{\text{w}}\underbracket{1}_{\text{v}}x_2$ or $C = x_1\underbracket{1}_{\text{v}}\underbracket{0}_{\text{w}}x_2$, then $f_{i,v}\leq f_{i,w}$ for all $i\notin\{v,w\}$, where the equality holds if and only if $C = \underbracket{0}_{\text{w}}\underbracket{1}_{\text{v}}x_2$.
		\item[(iii)] If $C = x_1\underbracket{0}_{\text{v}}111\ldots1\underbracket{0}_{\text{w}}x_2$, then $f_{i,v}<f_{i,w}$ for all $i\notin\{v,w\}$.
		\item[(iv)] If $C = x_1\underbracket{1}_{\text{w}}000\ldots0\underbracket{1}_{\text{v}}x_2$, then $f_{i,v}\leq f_{i,w}$ for all $i\notin\{v,w\}$, where the equality holds if and only if $C = x_1\underbracket{1}_{\text{w}}000\ldots0\underbracket{1}_{\text{v}}$ and $i$ precedes $w$.
	\end{itemize}
\end{theorem}
\begin{proof}
	Consider (i). Since $v$ and $w$ share the same neighbourhood, there exists an automorphism of $G$ that exchanges $v$ and $w$ and fixes the remaining vertices. Hence, the conclusion follows. 
 
    The remaining statements follows from the fact with Lemma~\ref{lem: f(iv)<f(iw)} that for each hypothesis, $N(w)$ is a subset of $N(v)$. Moreover, one can examine the equality conditions for (ii) and (iv) by Lemma~\ref{lem: f(iv)<f(iw)}. 
\end{proof}

\begin{corollary} \label{Corollary_F}
	Let $G$ be a threshold graph with code $C= \mathbf{0}^{s_1}\mathbf{1}^{t_1}\mathbf{0}^{s_2}\mathbf{1}^{t_2}\ldots \mathbf{0}^{s_k}\mathbf{1}^{t_k}$. Suppose that for $\ell=1,\dots,k,$ $v_\ell$ is a vertex corresponding to some zero in $\mathbf{0}^{s_\ell}$, and $w_\ell$ is a vertex corresponding to some one in $\mathbf{1}^{t_\ell}$. Then the following hold:
	\begin{itemize}
		\item Let $i$ be a vertex corresponding to a zero in $\mathbf{0}^{s_{\alpha}}$ for some $1\leq \alpha\leq k$. Then
		\begin{align*} 
			0 = f_{i,i}&<f_{i,w_k} \leq f_{i,w_{k-1}}<\cdots<f_{i,w_1}\leq f_{i,v_1}<\cdots<f_{i,v_{\alpha-1}} \\&<\underbracket{f_{i,v_{\alpha}}}_{\text{if it exists, \textit{i.e.}}\; s_{\alpha}\geq 2\;\text{and}\; i\neq v_\alpha}<f_{i,v_{\alpha+1}}<\cdots<f_{i,v_{k}}.
		\end{align*}
		\item Let $i$ be a vertex corresponding to a one in $\mathbf{1}^{t_{\alpha}}$ for some $1\leq \alpha\leq k$. Then
		\begin{align*} 
			0 = f_{i,i}&<f_{i,w_k} \leq f_{i,w_{k-1}}<\cdots<f_{i,w_{\alpha+1}}<\underbracket{f_{i,w_{\alpha}}}_{\text{if it exists, \textit{i.e.}}\;t_{\alpha}\geq 2\;\text{and}\; i\neq w_\alpha} <f_{i,w_{\alpha-1}}<\cdots< \\ &<f_{i,w_{1}}\leq f_{i,v_1}<f_{i,v_2}<\ldots<f_{i,v_k}.
		\end{align*}
	\end{itemize}
\end{corollary}

\begin{corollary} \label{Corollary_ordering_F_degrees}
	Let $G$ be a threshold graph, and let $i, w, v\in V(G)$ be distinct vertices. Then \begin{align*}
		f_{i,w} \geq f_{i.v}\quad \text{if and only if} \quad d_w \leq d_v.  
	\end{align*}
	This implies that the smallest non-zero entry $f_{p,q}$ is attained if and only if $p$ and $q$ are of the largest two degrees; and the largest non-zero entry $f_{r,s}$ is attained if and only if $r$ and $s$ are of the smallest two degrees.
\end{corollary}
\begin{proof}
	Let $C= \mathbf{0}^{s_1}\mathbf{1}^{t_1}\mathbf{0}^{s_2}\mathbf{1}^{t_2}\ldots \mathbf{0}^{s_k}\mathbf{1}^{t_k}$ be the code associated to $G$. For $\ell=1,\dots,k,$ let $v_\ell$ be a vertex corresponding some zero in $\mathbf{0}^{s_\ell}$, and $w_\ell$ be a vertex corresponding some one in $\mathbf{1}^{t_\ell}$. Then $d_{v_\ell} = \sum_{i=\ell}^k t_i$ and $d_{w_\ell} = -1 + \sum_{i=1}^k t_i + \sum_{i=1}^\ell s_i$. It follows that
	\begin{align}\label{eqn:ordering of degree}
		d_{v_k}< d_{v_{k-1}}<\cdots <d_{v_1}\leq d_{w_1}<d_{w_2}<\cdots<d_{w_k}
	\end{align}
	where the equality for $d_{v_1}$ and $d_{w_1}$ holds if and only if $s_1=1$. From Corollary~\ref{Corollary_F}, the conclusion follows.
\end{proof}

\begin{remark}\label{rmk:entries r}
	We note that $f_{v,w} = \tau r_{v,w}$ (where $\tau$ is the number of spanning trees). Hence, analogous results for effective resistance distance as in Theorem~\ref{thm: entries f} and Corollaries~\ref{Corollary_F} and \ref{Corollary_ordering_F_degrees} follow.
\end{remark}



\subsection{Linear orderings of the moments and accessibility indices via block structure}
\label{Linear orderings of the moments and accessibility indices via block structure}

Let $G$ be a threshold graph with code $C = \mathbf{0}^{s_1}\mathbf{1}^{t_1}\mathbf{0}^{s_2}\mathbf{1}^{t_2}\ldots \mathbf{0}^{s_k}\mathbf{1}^{t_k}$. It follows from (i) of Theorem~\ref{thm: entries f} that moments of vertices in the same block in $C$ are the same and so are their accessibility indices. Hence, we may define the moment of vertex corresponding to $0$ in $\mathbf{0}^{s_i}$ (resp. $1$ in $\mathbf{1}^{t_i}$) as $\mu(\mathbf{0}^{s_i})$ (resp. $\mu(\mathbf{1}^{t_i})$). Similarly, let $\alpha(\mathbf{0}^{s_i})$ (resp. $\alpha(\mathbf{1}^{t_i})$) denote the accessibility index of vertex corresponding to $0$ in $\mathbf{0}^{s_i}$ (resp. $1$ in $\mathbf{1}^{t_i}$).

Here is the main result of this section.
\begin{theorem}\label{thm:ordering accessibility}
	Let $G$ be a threshold graph with code $C = \mathbf{0}^{s_1}\mathbf{1}^{t_1}\mathbf{0}^{s_2}\mathbf{1}^{t_2}\ldots \mathbf{0}^{s_k}\mathbf{1}^{t_k}$. Then
	\begin{align*}
		\mu(\mathbf{0}^{s_k}) > \mu(\mathbf{0}^{s_{k-1}}) > \ldots > \mu(\mathbf{0}^{s_1}) \ge \mu(\mathbf{1}^{t_1}) > \mu(\mathbf{1}^{t_2}) \ldots > \mu(\mathbf{1}^{t_k})
	\end{align*}
	with equality if and only if $s_1 = 1$. This implies that
	\begin{align*}
		\alpha(\mathbf{0}^{s_k}) > \alpha(\mathbf{0}^{s_{k-1}}) > \ldots > \alpha(\mathbf{0}^{s_1})\ge \alpha(\mathbf{1}^{t_1}) > \alpha(\mathbf{1}^{t_2}) \ldots > \alpha(\mathbf{1}^{t_k})
	\end{align*}
	with equality if and only if $s_1 = 1$.
\end{theorem}

From \eqref{eqn:ordering of degree}, we immediately obtain the following characterization.
\begin{corollary}
	Let $G$ be a threshold graph and let $v, w\in V(G)$ be two distinct vertices. Then
	\[\alpha(v)> \alpha(w) \iff \mu(v) > \mu(w) \iff d_v < d_w.\]
\end{corollary}
%

\noindent
\textbf{Proof of Theorem~\ref{thm:ordering accessibility}:}
We begin with introducing notation. For $\ell=1,\dots,k,$ let $v_\ell$ be the vertex corresponding to the first zero in $\mathbf{0}^{s_\ell}$, and let $w_\ell$ be the vertex corresponding to the first one in $\mathbf{1}^{t_\ell}$. Define $\bar{v}_\ell$ to be the vertex of $G$ as follows: if $s_\ell\geq 2$ then $\bar{v}_\ell$ corresponds to the second zero in $\mathbf{0}^{s_\ell}$; and if $s_\ell= 1$ then $\bar{v}_\ell = v_\ell$. Similarly, we define $\bar{w}_\ell$ to be the vertex of $G$ as follows: if $t_\ell\geq 2$ then $\bar{w}_\ell$ corresponds to the second one in $\mathbf{1}^{t_\ell}$; and if $t_\ell= 1$ then $\bar{w}_\ell = w_\ell$.

Suppose that $s_1=1$. Since $\bar{v}_1$ and any vertex $x$ corresponding to one in $\mathbf{1}^{t_1}$ have the same neighbourhood, there exists an automorphism of $G$ that maps $\bar{v}_1$ to $x$. Hence, $\mu(\mathbf{0}^{s_1}) = \mu(\mathbf{1}^{t_1})$. 

We now consider moments of $\bar{v}_\ell$ and $\bar{w}_\ell$.  We observe that  $r_{x_1,\bar{v}_\ell} = r_{x_2,\bar{v}_\ell}$ if $x_1$ and $x_2$ are in the same cell with $x_1\ne \bar{v}_\ell$ and $x_2\ne \bar{v}_\ell$, and that if $s_\ell=1$ then  $r_{v_\ell,\bar{v}_\ell} = 0$, and if $t_\ell=1$ then $r_{w_\ell,\bar{w}_\ell} = 0$.  Then we can see
\begin{align*}
	\mu(\bar{v}_\ell) =&~\sum_{i=1}^n d_i r_{i,\bar{v}_\ell}= \sum\limits_{i = 1}^k s_id_{v_i}r_{v_i,\bar{v}_\ell}+\sum\limits_{i = 1}^k t_id_{w_i}r_{w_i,\bar{v}_\ell} - d_{v_\ell} r_{v_\ell,\bar{v}_\ell},\\
	\mu(\bar{w}_\ell) =&~\sum_{i=1}^n d_i r_{i,\bar{w}_\ell}= \sum\limits_{i = 1}^k s_id_{v_i}r_{v_i,\bar{w}_\ell}+\sum\limits_{i = 1}^k t_id_{w_i}r_{w_i,\bar{w}_\ell} - d_{w_\ell} r_{w_\ell,\bar{w}_\ell}.
\end{align*}
Let $j = 1,\dots, k-1$. We can find 
\begin{align*}\nonumber
	&\mu(\bar{v}_{j+1})-\mu(\bar{v}_{j}) \\\nonumber
	= &~\sum\limits_{i = 1}^k s_id_{v_i}(r_{v_i,\bar{v}_{j+1}}-r_{v_i,\bar{v}_j})+\sum\limits_{i = 1}^k t_id_{w_i}(r_{w_i,\bar{v}_{j+1}}-r_{w_i,\bar{v}_j}) + (d_{v_j} r_{v_j,\bar{v}_j} - d_{v_{j+1}} r_{v_{j+1},\bar{v}_{j+1}})\\
	& \hspace{80mm}\text{(by Theorem~\ref{thm: entries f} with Remark~\ref{rmk:entries r})}\\
	>&~ - d_{v_{j+1}} r_{v_{j+1},\bar{v}_{j+1}}+\sum\limits_{i = 1}^k s_id_{v_i}(r_{v_i,\bar{v}_{j+1}}-r_{v_i,\bar{v}_j}) \\
	>&~ - d_{v_{j+1}} r_{v_{j+1},\bar{v}_{j+1}} + d_{v_j}(r_{v_j,\bar{v}_{j+1}}-r_{v_j,\bar{v}_j}) +d_{v_{j+1}}(r_{v_{j+1},\bar{v}_{j+1}}-r_{v_{j+1},\bar{v}_j})\\
	=&~r_{v_j,\bar{v}_{j+1}} (d_{v_j}-d_{v_{j+1}})\hspace{35mm}\text{(because $r_{v_j,\bar{v}_{j+1}}\geq 0$ and $d_{v_j}-d_{v_{j+1}}>0$)}\\
	\geq&~0.
\end{align*}
Similarly, one can verify that $\mu(\bar{w}_{j})-\mu(\bar{w}_{j+1})>0$. 

Finally, we consider $\mu(\bar{v}_1)-\mu(\bar{w}_1)$ when $s_1 = 1$ or $s_1\geq 2$. If $s_1 \geq 2$, then by (ii) of Theorem~\ref{thm: entries f}, we have
\begin{align*}
	&\mu(\bar{v}_1)-\mu(\bar{w}_1) \\
	=&~\sum\limits_{i = 1}^k s_id_{v_i}(r_{v_i,\bar{v}_1}-r_{v_i,\bar{w}_1})+\sum\limits_{i = 1}^k t_id_{w_i}(r_{w_i,\bar{v}_1}-r_{w_i,\bar{w}_1}) +(d_{w_1} r_{w_1,\bar{w}_1}- d_{v_1} r_{v_1,\bar{v}_1})\\
	>&~ d_{v_1}(r_{v_1,\bar{v}_1}-r_{v_1,\bar{w}_1})+d_{w_1}(r_{w_1,\bar{v}_1}-r_{w_1,\bar{w}_1})+ (d_{w_1} r_{w_1,\bar{w}_1}- d_{v_1} r_{v_1,\bar{v}_1})\\
	=&~r_{w_1,\bar{v}_1}(d_{w_1}-d_{v_1})\\
	\geq &~ 0.
\end{align*}
Using (ii) of Theorem~\ref{thm: entries f}, one can show that if $s_1 = 1$ then $\mu(\bar{v}_1)-\mu(\bar{w}_1) = 0$. 

Therefore, the theorem follows. \hfill $\Box$



\section*{Acknowledgement}
The authors are grateful to Ada Chan at York University for constructive conversations during this project, and to the Fields Institute for Research in the Mathematical Sciences for hosting the 2023 Fields Undergraduate Summer Research Program (FUSRP).

\section*{Funding}
J. Breen is supported by the Natural Sciences and Engineering Research Council of Canada (NSERC) Grant RGPIN-2021-03775. S. Kim is supported in part by funding from the Fields Institute for Research in Mathematical Sciences and from NSERC. A. Low Fung, A. Mann, A.~A. Parfeni, and G.~Tedesco were supported by funding from the Fields Institute and by funding from NSERC Grants RGPIN-2021-03775 and RGPIN-2021-03609.


\begin{thebibliography}{10}
	
	\bibitem{altafini2023edge}
	Diego Altafini, Dario~A Bini, Valerio Cutini, Beatrice Meini, and Federico
	Poloni.
	\newblock An edge centrality measure based on the {K}emeny constant.
	\newblock {\em SIAM Journal on Matrix Analysis and Applications},
	44(2):648--669, 2023.
	
	\bibitem{banerjee2017normalized}
	Anirban Banerjee and Ranjit Mehatari.
	\newblock On the normalized spectrum of threshold graphs.
	\newblock {\em Linear Algebra and its Applications}, 530:288--304, 2017.
	
	\bibitem{bapat2013adjacency}
	R.~B. Bapat.
	\newblock On the adjacency matrix of a threshold graph.
	\newblock {\em Linear Algebra and its Applications}, 439(10):3008--3015, 2013.
	
	\bibitem{bapat2010graphs}
	Ravindra~B. Bapat.
	\newblock {\em Graphs and Matrices}, volume~27.
	\newblock Springer, 2010.
	
	\bibitem{breen2019computing}
	Jane Breen, Steve Butler, Nicklas Day, Colt DeArmond, Kate Lorenzen, Haoyang
	Qian, and Jacob Riesen.
	\newblock Computing {K}emeny's constant for barbell-type graphs.
	\newblock {\em The Electronic Journal of Linear Algebra}, 35:583--598, 2019.
	
	\bibitem{breen2022kemeny}
	Jane Breen, Emanuele Crisostomi, and Sooyeong Kim.
	\newblock {K}emeny’s constant for a graph with bridges.
	\newblock {\em Discrete Applied Mathematics}, 322:20--35, 2022.
	
	\bibitem{fanchung}
	Fan R.~K. Chung.
	\newblock {\em Spectral Graph Theory}, volume~92.
	\newblock American Mathematical Soc., 1997.
	
	\bibitem{ciardo2022kemeny}
	Lorenzo Ciardo, Geir Dahl, and Steve Kirkland.
	\newblock On {K}emeny's constant for trees with fixed order and diameter.
	\newblock {\em Linear and Multilinear Algebra}, 70(12):2331--2353, 2022.
	
	\bibitem{faught20221}
	Nolan Faught, Mark Kempton, and Adam Knudson.
	\newblock A 1-separation formula for the graph {K}emeny constant and {B}raess
	edges.
	\newblock {\em Journal of Mathematical Chemistry}, 60(1):49--69, 2022.
	
	\bibitem{hammer1996laplacian}
	Peter~L. Hammer and Alexander~K. Kelmans.
	\newblock Laplacian spectra and spanning trees of threshold graphs.
	\newblock {\em Discrete Applied Mathematics}, 65(1-3):255--273, 1996.
	
	\bibitem{hornjohnson}
	Roger~A. Horn and Charles~R. Johnson.
	\newblock {\em Matrix Analysis}.
	\newblock Cambridge University Press, 2nd edition, 2012.
	
	\bibitem{jacobs2015eigenvalues}
	David~P. Jacobs, Vilmar Trevisan, and Fernando Tura.
	\newblock Eigenvalues and energy in threshold graphs.
	\newblock {\em Linear Algebra and its Applications}, 465:412--425, 2015.
	
	\bibitem{jang2023kemeny}
	Jihyeug Jang, Sooyeong Kim, and Minho Song.
	\newblock {K}emeny's constant and {W}iener index on trees.
	\newblock {\em Linear Algebra and its Applications}, 2023.
	
	\bibitem{kemenysnell}
	John~G. {K}emeny and J.~Laurie Snell.
	\newblock {\em Finite {M}arkov {C}hains}.
	\newblock The University Series in Undergraduate Mathematics. D. Van Nostrand
	Co., Inc., Princeton, N.J.-Toronto-London-New York, 1960.
	
	\bibitem{kim2023bounds}
	Sooyeong Kim, Neal Madras, Ada Chan, Mark Kempton, Stephen Kirkland, and Adam
	Knudson.
	\newblock Bounds on {K}emeny's constant of a graph and the {N}ordhaus-{G}addum
	problem.
	\newblock {\em arXiv preprint arXiv:2309.05171}, 2023.
	
	\bibitem{kirkland2016random}
	Steve Kirkland.
	\newblock Random walk centrality and a partition of {K}emeny’s constant.
	\newblock {\em Czechoslovak Mathematical Journal}, 66:757--775, 2016.
	
	\bibitem{kirkland2016kemeny}
	Steve Kirkland and Ze~Zeng.
	\newblock {K}emeny's constant and an analogue of {B}raess' paradox for trees.
	\newblock {\em Electronic Journal of Linear Algebra}, 31:444--464, 2016.
	
	\bibitem{levene2002kemeny}
	Mark Levene and George Loizou.
	\newblock {K}emeny's constant and the random surfer.
	\newblock {\em American Mathematical Monthly}, 109(8):741--745, 2002.
	
	\bibitem{thresholdbook}
	N.~V.~R. Mahadev and U.~N. Peled.
	\newblock {\em Threshold Graphs and Related Topics}.
	\newblock North Holland, 1995.
	
	\bibitem{noh2004random}
	Jae~Dong Noh and Heiko Rieger.
	\newblock Random walks on complex networks.
	\newblock {\em Physical review letters}, 92(11):118701, 2004.
	
	\bibitem{palacios2011broder}
	Jos{\'e}~Luis Palacios and Jos{\'e}~M. Renom.
	\newblock Broder and {K}arlin's formula for hitting times and the {K}irchhoff
	index.
	\newblock {\em International Journal of Quantum Chemistry}, 111(1):35--39,
	2011.
	
	\bibitem{shapiro1987electrical}
	Louis~W. Shapiro.
	\newblock An electrical lemma.
	\newblock {\em Mathematics Magazine}, 60(1):36--38, 1987.
	
	\bibitem{taittobin}
	Michael Tait and Josh Tobin.
	\newblock Three conjectures in extremal spectral graph theory.
	\newblock {\em Journal of Combinatorial Theory, Series B}, 126:137--161, 2017.
	
	\bibitem{wang2017kemeny}
	Xiangrong Wang, Johan L.~A. Dubbeldam, and Piet Van~Mieghem.
	\newblock {K}emeny's constant and the effective graph resistance.
	\newblock {\em Linear Algebra and its Applications}, 535:231--244, 2017.
	
\end{thebibliography}

\end{document}